\documentclass[12pt,leqno]{amsart}
\usepackage{amsmath,amsthm,amsfonts,amssymb, amscd,amsbsy,multirow, hyperref}
\usepackage{hyperref}
\usepackage{graphicx}
\usepackage{xcolor}
\graphicspath{ }
\usepackage{wrapfig}
\usepackage{accents}
\usepackage{caption}
\usepackage{subcaption}
\usepackage{calligra}

\numberwithin{figure}{section}

\theoremstyle{plain}
\newtheorem{thm}{Theorem}[section]
\newtheorem{lem}[thm]{Lemma}

\theoremstyle{definition}
\newtheorem{defn}{Definition}[section]

\theoremstyle{remark}
\newtheorem*{rem}{Remark}

\usepackage{mathtools}
\title[On gradient Einstein solitons]{Some characterizations of $\rho$-Einstein solitons}
\author[A. A. Shaikh, Antonio W. Cunha and P. Mandal]{Absos Ali Shaikh$^{*1}$, Antonio W. Cunha$^{2}$ and Prosenjit Mandal$^{3}$}
\address{$^1$Department of Mathematics,\newline University of
Burdwan, Golapbag,\newline Burdwan-713104,\newline West Bengal, India.}
\address{$^2$Departamento de Matem\'{a}tica,\newline Universidade Federal do Piau\'{\i}, \newline Teresina 64.049-550,\newline Piaui, Brazil.}
\address{$^3$Department of Mathematics,\newline University of
Burdwan, Golapbag,\newline Burdwan-713104,\newline West Bengal, India.}
\email{$^1$aask2003@yahoo.co.in, aashaikh@math.buruniv.ac.in}
\email{$^2$wilsoncunha@ufpi.edu.br}
\email{$^3$prosenjitmandal235@gmail.com}
\begin{document}
\begin{abstract}
In this article we have showed that a gradient $\rho$-Einstein soliton with a vector field of bounded norm and satisfying some other conditions is isometric to the Euclidean sphere. Later, we have proved that a non-trivial complete gradient $\rho$-Einstein soliton with finite weighted Dirichlet integral and certain restriction on Ricci curvature must be of constant scalar curvature and steady Ricci flat. Finally, we have proved that a non-shrinking or non-expanding gradient traceless Ricci soliton possessing some conditions must be steady.  
\end{abstract}
\noindent\footnotetext{
$\mathbf{2020}$\hspace{5pt}Mathematics\; Subject\; Classification: 53C20; 53C21; 53C25.} 
\maketitle

\section*{Introduction and preliminaries}
The Ricci flow introduced by Hamilton has been studied intensively in recent years and plays a key role in Perelman's proof of the Poincar\'e conjecture. It was studied in the sense to find critical points of scalar curvature functional
$$g\longmapsto E(g)=\int_MRdv_g,$$
where $M$ is a $n(\geq3)$- dimensional compact smooth manifold, $g$ is a Riemannian metric on $M$ and $R$ is its scalar curvature. The existence of such metrics, even for short times, was one of the main reasons which led Hamilton to introduce the Ricci flow 
$$\frac{\partial g}{\partial t} =-2Ric.$$ 
For an study to Ricci flow, we refer the reader to \cite{Chow-B}.

An important aspect in the treatment of the Ricci flow is the study of Ricci solitons. A Ricci soliton is a $n(\geq2)$-dimensional Riemannian manifold $(M, g)$, endowed with a smooth vector field $X$ satisfying
 \begin{equation}\label{Ricci soliton}
 {\rm Ric} +\frac{1}{2}\mathcal{L}_Xg = \lambda g,
 \end{equation}
where $\lambda$ is a constant and $\mathcal{L}$ stands for the Lie derivative. Ricci solitons model the formation of singularities in the Ricci flow and they correspond to self-similar solutions, i.e., they are stationary points of this flow in the space of metrics modulo diffeomorphisms and scalings, see \cite{hamilton} for more details. Thus, classification of Ricci solitons or understanding their geometry is definitely an important issue.

If $X$ is the gradient of a smooth function $f$ on $M$, then such a Ricci soliton is called gradient Ricci soliton. In this case, (\ref{Ricci soliton}) becomes
\begin{equation}\label{Grad Ricci soliton}
 {\rm Ric} + {\rm Hess}(f) = \lambda g,
 \end{equation}
where $Hess(f)$ stands for the Hessian of $f$. Also, a Ricci soliton $(M, g, X, \lambda)$ is called \textit{expanding}, \textit{steady} or \textit{shrinking} according as $\lambda < 0$, $\lambda = 0$ or $\lambda > 0$, respectively. After rescaling the metric $g$ we may assume that $\lambda \in \{-\frac{1}{2}, 0,\frac{1}{2}\}$. Moreover, when the vector field $X$ is \textit{trivial} or $f$ is constant then the Ricci soliton is called trivial. Therefore, Ricci solitons are natural extensions of Einstein manifolds. A gradient soliton is {\em rigid} if it is isometric to a quotient of $N\times\mathbb{R}^k$, where $N$ is an Einstein manifold and $f=\frac{\lambda}{2}|x|^2$ on the Euclidean factor.
In general, it is natural to consider geometric flows of the following type on a $n(\geq3)$- dimensional Riemannian manifold $(M,g)$:
$$\frac{\partial g}{\partial t}=-2({\rm Ric}-\rho Rg),$$
for some $\rho\in\mathbb{R}\diagdown\{0\}$. The parabolic theory, for these flows, was developed by Catino et. al. \cite{Catino13}, which was first considered by Bourguignon \cite{Bour}. They called such a flow as Ricci-Bourguignon flows. It is interesting to prove short time existence for every $-\infty<\rho<\frac{1}{2(n-1)}$. 
Associated to these flows, they defined the following notion of
 $\rho$-Einstein solitons.
 
 \begin{defn}
Let $(M,g)$ be a Riemannian manifold of dimension $n(\geq 3)$, and let $\rho\in\mathbb{R}$, $\rho\neq 0$. Then $M$ is called a $\rho$-Einstein soliton if there is a smooth vector field $X$ such that

\begin{equation}\label{a2}
{\rm Ric}+\frac{1}{2}\mathcal{L}_Xg-\rho Rg=\lambda g,
\end{equation}
where $Ric$ is the Ricci curvature tensor, $\lambda$ is a constant and  $\mathcal{L}_Xg$ represents the Lie derivative of g in the direction of the vector field $X$. Throughout the paper such
a $\rho$-Einstein soliton will be denoted by $(M, g, X, \rho)$.
\end{defn}
If there exists a smooth function $f:M\rightarrow\mathbb{R}$ such that $X=\nabla f$ then the $\rho$-Einstein soliton is called a gradient $\rho$-Einstein soliton, denoted by $(M,g,f,\rho)$ and in this case (\ref{a2}) takes the form 
 
\begin{equation}\label{ro soliton}
{\rm Ric}+{\rm Hess}(f)-\rho Rg=\lambda g.
\end{equation}

As usual, a $\rho$-Einstein soliton is called steady for $\lambda=0$, shrinking for $\lambda>0$ and expanding for $\lambda<0$. The function $f$ is called a $\rho$-Einstein potential of the gradient $\rho$-Einstein soliton.

For special values of the parameter $\rho$, a $\rho$-Einstein soliton is called 
\begin{enumerate}
\item[i)] gradient Einstein soliton if $\rho=\frac{1}{2}$,
\item[ii)] gradient traceless Ricci soliton if $\rho=\frac{1}{n}$,
\item[iii)] gradient Schouten soliton if $\rho=\frac{1}{2(n-1)}$.
\end{enumerate}
Later, this notion has been generalized in various directions such as $m$-quasi Einstein manifold \cite{HLX2015}, $(m,\rho)$-quasi Einstein manifold \cite{HW2013}, Ricci-Bourguignon almost
soliton \cite{DW2018} etc.

 Also they proved that every compact gradient Einstein, Schouten or traceless Ricci soliton is trivial. Furthermore they classified three-dimensional gradient shrinking Schouten soliton and proved that it is isometric to a finite quotient of either $\mathbb{S}^3$ or $\mathbb{R}^3$ or $\mathbb{R}\times\mathbb{S}^2$. In \cite{Huang} Huang deduced a sufficient condition for a compact gradient shrinking $\rho$-Einstein soliton to be isometric to a quotient of the round sphere $\mathbb{S}^n$. Recently in 2019, Mondal and Shaikh \cite{Mondal} showed that a compact gradient $\rho$-Einstein soliton with $\nabla f$ a non-trivial conformal vector field, is isometric to the Euclidean sphere $\mathbb{S}^n$. Further Dwivedi \cite{DW2018} proved the following isometry theorem for gradient Ricci-Bourguignon soliton.
 
 \begin{thm}{\cite{DW2018}}\label{thmDwi}
 A non-trivial compact gradient Ricci-Bourguignon soliton is isometric to an Euclidean sphere if any one of the following holds:
 \begin{enumerate}
 \item[(i)] $M$ has constant scalar curvature,
 \item[(ii)] $\int_M\langle\nabla R, \nabla f\rangle\leq0,$
 \item[(iii)] $M$ is a homogeneous manifold.
 \end{enumerate}
 \end{thm}
  Recently as a consequence of Theorem \ref{thmDwi} Shaikh et al. \cite{Absos} proved the following.
 \begin{thm}{\cite{Absos}}\label{Absos}
 A non-trivial compact gradient $\rho$-Einstein soliton $(M,g,f,\rho)$ has constant scalar curvature and therefore $M$ is isometric to the Euclidean sphere.
 \end{thm}
 
 \indent This paper is organized in the following way: In section $1$ we have proved that  a gradient $\rho$-Einstein soliton endowed with a vector field of bounded norm and some other conditions is isometric to the Euclidean sphere $\mathbb{S}^n$. Also we have proved that a non-trivial complete gradient $\rho$-Einstein soliton with finite weighted Dirichlet integral and certain condition on Ricci curvature must be of constant scalar curvature and steady Ricci flat. In the last section we have showed that a non-shrinking or non-expanding gradient traceless Ricci soliton with some conditions must be steady. 
 
\section{$\rho$-Einstein solitons isometric to the Euclidean sphere}
We start this section with the following compactness lemma. 

\begin{lem}\label{lemma mandal}
Let $(M, g, X, \rho)$ be a $\rho$-Einstein soliton with $\|X\|$ bounded and $\rho R \geq k$, for some real constant k such that $(\lambda+k)>0$. Then $M$ is compact.
\end{lem}
\begin{proof}
Let $o$ be a point in $M$ and consider any geodesic $\gamma : [0,\infty ) \rightarrow M $ emanating from $o$ and parametrized by arc length $s$. Then along $\gamma$ we obtain 
\begin{equation*}
\mathcal{L}_X g(\gamma', \gamma') = 2g(\nabla_{\gamma'}X,\gamma')=2 \frac{d}{ds}[g(X,\gamma')].
\end{equation*}
Thus from (\ref{a2}) and using the Cauchy-Schwarz inequality, we have
\begin{eqnarray*}
\int_{0}^{r} Ric(\gamma'(s), \gamma'(s)) ds &=& \lambda r + \int_{0}^{r} \rho R g + g(X_o,\gamma'(0))-g(X_{\gamma(r)}, \gamma'(r))\\
&\geq& \lambda r + kr+ g(X_o, \gamma'(0))-\|X_{\gamma(r)}\|.
\end{eqnarray*}
Now, since $\|X\|$ is bounded, we obtain that
\begin{equation*}
\int_{0}^{+\infty} Ric(\gamma'(s),\gamma'(s))= + \infty,
\end{equation*} 
and from Ambrose's compactness criteria \cite {WA1957} it follows that $M$ is compact.
\end{proof}
\begin{thm}
Let $(M, g, f, \rho)$ be a non-trivial gradient $\rho$-Einstein soliton with $\|\nabla f\|$ bounded and $\rho R \geq k$, for some real constant k such that $(\lambda+k)>0$. Then $M$ is isometric to the Euclidean sphere.
\end{thm}
\begin{proof}
From Lemma \ref{lemma mandal} it follows that $M$ is compact. Now from Theorem \ref{Absos} we get that the scalar curvature $R$ is constant and therefore from Theorem \ref{thmDwi}, we obtain the desired result.
\end{proof}

The following results are well known to gradient $\rho$-Einstein solitons.

\begin{lem}{\cite{Catino14}}\label{lemma1}
Let $(M, g, f,\rho)$ be a gradient $\rho$-Einstein soliton. Then the following identities hold:
\begin{equation}\label{eq3}
 \Delta f=(n\rho-1)R+n\lambda,
\end{equation}
\begin{equation}\label{gradescalar}
(1-2(n-1)\rho)\nabla R=2{\rm Ric}(\nabla f,\cdot),
\end{equation}
\begin{equation}\label{flaplaR}
(1-2(n-1)\rho)\Delta R=\langle\nabla R,\nabla f\rangle+2\lambda R+2(\rho R^2-|{\rm Ric}|^2+\lambda R).
\end{equation}
\end{lem}

The next lemma gives a Bochner type formula for gradient $\rho$-Einstein solitons.

\begin{lem}\label{lemma6}
Let $(M, g,  f, \rho)$ be a complete gradient $\rho$-Einstein soliton with \linebreak$\rho\not=\frac{1}{2(n-1)}$. Then
\begin{equation}
\frac{1}{2}\Delta|\nabla f|^2=\frac{2\rho-1}{1-2(n-1)\rho}{\rm Ric}(\nabla f, \nabla f)+|Hess(f)|^2.
\end{equation}
\end{lem}
\begin{proof}
From Bochner formula and equation \eqref{eq3} we have
$$\frac{1}{2}\Delta|\nabla f|^2=Ric(\nabla f, \nabla f)+(n\rho-1)\langle\nabla f, \nabla R\rangle+|Hess(f)|^2.$$
Now using equation \eqref{gradescalar} we obtain
\begin{eqnarray}
\frac{1}{2}\Delta|\nabla f|^2 &=& Ric(\nabla f, \nabla f)+(n\rho-1)\langle\nabla f, \nabla R\rangle+|Hess(f)|^2\\
&=&\left[\frac{2(n\rho-1)}{1-2(n-1)\rho}+1\right]Ric(\nabla f,\nabla f)+|Hess(f)|^2\\
&=& \left[\frac{2\rho-1}{1-2(n-1)\rho}\right]Ric(\nabla f,\nabla f)+|Hess(f)|^2.
\end{eqnarray}
\end{proof}
\begin{thm}
Let $(M, g,  f,\rho)$ be a non-trivial complete gradient $\rho$-Einstein soliton. Suppose that the potential function has finite weighted Dirichlet integral, i.e.,
\begin{equation}
\int_{M\backslash B(q,r)}d(x,q)^{-2}|\nabla f|^2<\infty,
\end{equation}
where $B(q,r)$ is a ball with radius $r>0$ and centre at $q$ and $d(x,q)$ is the distance function
from some fixed point $q\in M$. If any one of the following
\begin{enumerate}
\item[i)] the Ricci curvature is non-positive and either $\rho>\frac{1}{2}$ or $\rho<\frac{1}{2(n-1)}$,
\item[ii)] the Ricci curvature is non-negative and  $\frac{1}{2(n-1)}<\rho<\frac{1}{2}$,
\end{enumerate}
holds, then the scalar curvature $R$ is constant and $M$ is steady Ricci flat.
\end{thm}
\begin{proof}
Let us consider the cut-off function, introduced in \cite{Cheeger}, $\phi_r\in C_0^\infty(B(q,2r))$ for $r>0$, such that
\begin{eqnarray}\label{prob 0}\left\{ \begin{array}{lllll}
0\leq\phi_r\leq1 & {\rm in}\,\,B(q,2r)\\
\,\,\,\,\,\phi_r=1 & {\rm in}\,\, B(q,r)\\
\,\,\,\,\,|\nabla\phi_r|^2\leq\frac{C_1}{r^2} & {\rm in}\,\,B(q,2r)\\
\,\,\,\,\,\Delta\phi_r\leq\frac{C_1}{r^2} & {\rm in}\,\,B(q,2r),\\
\end{array}\right.
\end{eqnarray}
where $C_1>0$ is a constant. From Lemma \ref{lemma6} we obtain
\begin{equation}
\int_M|Hess( f)|^2\phi_r^2+\frac{2\rho-1}{1-2(n-1)\rho}\int_MRic(\nabla f,\nabla f)\phi_r^2=\int_M\frac{1}{2}\Delta|\nabla f|^2\phi_r^2.
\end{equation}
Now using integration by parts and our assumption, we have
$$\int_M\frac{1}{2}\Delta|\nabla f|^2\phi_r^2=\int_M\frac{1}{2}|\nabla f|^2\Delta\phi_r^2\leq\int_{{M}\backslash B(q,r)}\frac{C_1}{2r^2}|\nabla f|^2\rightarrow0,$$
as $r\rightarrow\infty$. Hence, we obtain
$$\int_M|Hess( f)|^2+\frac{2\rho-1}{1-2(n-1)\rho}\int_MRic(\nabla f,\nabla f)=0.$$
Since in both cases $(i)$ and $(ii)$,
$$\frac{2\rho-1}{1-2(n-1)\rho}Ric(\nabla f,\nabla f)\geq0,$$
thus we obtain $Hess( f)=0$ and $Ric(\nabla f,\nabla f)=0$. Substituting this in the structural equation, we obtain
$$(\lambda+\rho R)|\nabla f|^2=0.$$
Therefore, $R=-\frac{\lambda}{\rho}$ and from structural equation we get $Ric=0$. Finally equation \eqref{flaplaR} yields $\lambda=0$.
\end{proof}
\begin{rem}
In the above theorem the condition $\rho=\frac{1}{2(n-1)}$ was observed  by Catino et. all \cite{Catino-Mazzieri}. They  proved that every complete gradient steady Schouten soliton is trivial, hence Ricci flat.
\end{rem}

\section{Steady gradient $\rho$-Einstein solitons}
First we state some results as lemmas which will be used in the sequel. The first one is due to Caminha et al. \cite{Caminha:10} (for more details, see Proposition 1 of \cite{Caminha:10}). 
\begin{lem}\label{lemma Caminha}
Let $X$ be a smooth vector field on the $n$-dimensional, complete, non-compact, oriented Riemannian manifold $M$, such that ${\rm div}_M X$ does not change sign on $M$. If $|X| \in L^1(M)$, then ${\rm div}_M X = 0$.
\end{lem}

Here we use the notation $L^p(M) = \{u : M \rightarrow \mathbb{R} \ : \int_{M}|u|^p dM < +\infty\}$ for each real $p \geq 1$.

The second and third one are due to Yau, corresponding to Theorem 3 of \cite{Yau:76} and Schoen and Yau \cite{SC2010} respectively.

\begin{lem}\label{lema Yau}
Let $u$ be a non-negative smooth subharmonic function on a complete Riemannian manifold $M$ of dimension $n$. If $u\in L^p(M)$, for some $p > 1$, then $u$ is constant.
\end{lem}
\begin{lem}\label{a10}\cite{SC2010}
If $f$ is a non-negative subharmonic function in $B(q,2r)$ contained in $M$, then the following inequality holds:
 \begin{equation}\label{a11}
 \int_{B(q,r)}|\nabla f|^2 \leq\frac{C}{r^2}\int_{B(q,2r)}f^2,
 \end{equation}
 where $B(q,r)$ is a ball with radius $r>0$ and center at $q$ and C is a real constant.
\end{lem}

For the next result we note that a Riemannian manifold $M$ of dimension $n$ is parabolic if every subharmonic functions $u$ on $M$ with $u^\ast=\sup_Mu<+\infty$ must be a constant. Equivalently, if any positive superharmonic function $u$ ($i.e, \Delta u\leq 0$) is constant. For more details, see \cite{Grigoryan}.

\begin{thm}
Let $(M,g,f,\frac{1}{n})$ be a non-shrinking gradient traceless Ricci soliton with $f\geq k_1,$ for some real constant $k_1>0$. If any one of the following
\begin{enumerate}
\item[(i)] $M$ is parabolic,
\item[(ii)]$|\nabla f|\in L^1(M)$,
\item[(iii)] $\frac{1}{f}\in L^p(M)$ for some $p>1$,
\item[(iv)] $M$ be of linear volume growth,
\end{enumerate}
holds, then $M$ must be steady.
\end{thm} 
\begin{proof}
 Taking the trace of equation $\eqref{ro soliton}$ we get
 \begin{equation}\label{a8}
 R+\Delta f=\lambda n+\rho Rn.
 \end{equation}
 On a gradient traceless Ricci soliton $\rho=\frac{1}{n}$ and hence from (\ref{a8}) it follows that
 \begin{equation}\label{a9}
 \Delta f = \lambda n\leq0,
  \end{equation}
i.e., $f$ is superharmonic. So, if (i) holds  we obtain that $f$ is constant, and consequently from \eqref{a9} we have $\lambda=0$. The case (ii) follows from Lemma \ref{lemma Caminha} and equation \eqref{a9}.

Now we observe that
 \begin{equation*}
  \Big(\frac{1}{f}\Big)_j = -\Big(\frac{1}{f^2}\Big)f_j,
 \end{equation*}
 and
 \begin{equation*}
  \Big(\frac{1}{f}\Big)_{jj} = \Big(\frac{2}{f^3}\Big)f^2_j -\Big(\frac{1}{f^2}\Big)f_{jj}.
 \end{equation*}
 Hence
 \begin{equation}\label{laplac}
 \Delta \Big(\frac{1}{f}\Big)=\Big(\frac{2}{f^3}\Big) |\nabla f|^2 -\frac{\Delta f}{f^2}.
 \end{equation}
 Since $\Delta f\leq 0$, it follows that $\Delta (\frac{1}{f})\geq 0$.  So, if (iii) holds  we have from Lemma \ref{lema Yau} that $\frac{1}{f}$ is constant on $M$, thus $f$ is constant and from \eqref{a9} entail $\lambda=0$.

Finally, suppose (iv) holds. Then from equation \eqref{a9} we have $\Delta f $ $\leq 0$, thus from \eqref{laplac} it follows that $\Delta (\frac{1}{f})\geq 0$. Since $M$ is of linear volume growth, i.e., $V(B(q,r))\leq C_2r$, for some constant $C_2>0$, we obtain from Lemma \ref{a10} that
 \begin{eqnarray*}
 \int_{B(q,r)}|\nabla \frac{1}{f}|^2&&\leq \Big(\frac{C}{r^2}\Big)\int_{B(q,2r)}\Big(\frac{1}{f^2}\Big)\\
  && \leq \Big(\frac{C}{r^2k_1^2}\Big) V(B(q,2r)) \\
  &&\leq \Big(\frac{C}{r^2k_1^2}\Big)C_2 2r\\
 &&\leq \frac{2CC_2}{rk_1^2} \rightarrow 0  
 \end{eqnarray*}
 as $r\rightarrow \infty$. Therefore 
 \begin{equation}
  \int_{M}|\nabla \frac{1}{f}|^2=0,
 \end{equation}
 which follows that the function $\frac{1}{f}$ is constant, that is $f$ is constant. Therefore $\Delta f =0$. Thus from (\ref{a9}) we have $\lambda=0$. This completes the proof. 
\end{proof}

\begin{thm}
Let $(M,g,f, \frac{1}{n})$ be a non-expanding gradient traceless Ricci soliton with non-negative potential function $f$. If $f\in L^p(M)$ for some $p>1$, then $M$ must be steady.
\end{thm}
\begin{proof}
Since for non-expanding solitons $\lambda\geq 0$, it follows from equation (\ref{a9}) that $$\Delta f=\lambda n\geq 0,$$ 
i.e., $f$ is a non-negative subharmonic function. Hence from Lemma \ref{lema Yau} we get $f$ constant, and $0=\Delta f=\lambda n\geq0$. Therefore $\lambda=0$.
\end{proof}

Finally we consider a non-expanding gradient traceless Ricci  soliton and proved an inequality analogous to Lemma \ref{a10} for its potential function.

\begin{thm}
If $(M,g,f, \frac{1}{n})$ be a non-expanding gradient traceless Ricci soliton with non-negative potential function $f$, then $f$ satisfies the following integral inequality
\begin{equation*}
\int_{B(q,r)}|\nabla f|^2 \leq\frac{C}{r^2}\int_{B(q,2r)}f^2,
\end{equation*}
 where $B(q,r)$ is a ball with radius $r>0$ and center at $q$ and C is a real constant.
\end{thm}
\begin{proof}
Since for non-expanding solitons we have $\lambda\geq 0$, it follows from equation (\ref{a9}) that $$\Delta f\geq 0.$$ Thus $f$ is a non-negative subharmonic function. Therefore from lemma \ref{a10} we get our result.
\end{proof}
\section{acknowledgment}
 The second author is partially supported by CNPq, Brazil, grant 430998/2018-0 and FAPEPI (Edital 007-2018) and the third author gratefully acknowledges to the
 CSIR(File No.:09/025(0282)/2019-EMR-I), Govt. of India for financial assistance.

\end{document}